\documentclass[11pt]{article}
\usepackage{amssymb}
\usepackage{amsmath}
\usepackage{amsfonts}

\newtheorem{theorem}{Theorem}

\newtheorem{definition}[theorem]{Definition}
\newtheorem{example}[theorem]{Example}

\newtheorem{lemma}[theorem]{Lemma}

\newtheorem{remark}[theorem]{Remark}

\newenvironment{proof}[1][Proof]{\noindent\textbf{#1.} }{\ \rule{0.5em}{0.5em}}

\newcommand{\func}[1]{\operatorname{#1}}
\numberwithin{equation}{section}

\begin{document}

\title{Long-memory Gaussian processes governed by generalized Fokker-Planck
equations}
\author{Luisa Beghin \thanks{%
luisa.beghin@uniroma1.it, Department of Statistical Sciences, Sapienza -
University of Rome }}
\maketitle

\begin{abstract}
\ \newline
It is well-known that the transition function of the Ornstein-Uhlenbeck
process solves the Fokker-Planck equation. This standard setting has been
recently generalized in different directions, for example, by considering
the so-called $\alpha $-stable driven Ornstein-Uhlenbeck, or by
time-changing the original process with an inverse stable subordinator. In
both cases, the corresponding partial differential equations involve
fractional derivatives (of Riesz and Riemann-Liouville types, respectively)
and the solution is not Gaussian. We consider here a new model, which cannot
be expressed by a random time-change of the original process: we start by a
Fokker-Planck equation (in Fourier space) with the time-derivative replaced
by a new fractional differential operator. The resulting process is Gaussian
and, in the stationary case, exhibits long-range dependence. Moreover, we
consider further extensions, by means of the so-called convolution-type
derivative.

\end{abstract}

\section{Introduction and notation}

The Ornstein--Uhlenbeck process has been widely studied and applied in many
different fields, mainly in finance and physics. Indeed, its mean-reverting
property is useful, in particular, in the financial data description, since
many economic quantities (such as interest rates) appear to be pulled back
to some average value, in the long-term period. Moreover, it is Gaussian,
Markovian and has a stationary version, which is the only nontrivial process
satisfying all the previous conditions.

It is well-known that the density of the Ornstein-Uhlenbeck (hereafter OU)
process satisfies the so-called Fokker-Planck (FP) equation, i.e.%
\begin{equation}
\frac{\partial }{\partial t}u(x,t)=\eta \frac{\partial }{\partial x}%
[xu(x,t)]+D\frac{\partial ^{2}}{\partial x^{2}}u(x,t),\qquad x\in \mathbb{R},%
\text{ }t\geq 0,\text{ }D,\eta >0,  \label{fp}
\end{equation}%
with the initial condition $u(x,0)=\delta (x)$ and the boundary condition $%
\lim_{|x|\rightarrow +\infty }xu(x,t)=0,$ for any $t\geq 0$ (see, for
example, \cite{PAV}, p.95)$.$

Fractional extensions of the previous equation have been proposed by many
authors, in different directions: in particular, when a fractional
Riemann-Liouville derivative (with respect to time) is introduced in the
right-hand side, an anomalous diffusion is obtained, which can be described
as a random time-change of the OU process by means of the inverse of an
independent stable subordinator (see \cite{MAG} and \cite{MAG2}). In \cite%
{JAN} this model is compared to the OU process driven by the $\alpha $%
-stable L\'{e}vy motion (which is governed by equation (\ref{fp}) with the
second order derivative replaced by the Riesz fractional operator). Later,
in \cite{GAJ}, the model has been generalized to the case where the
time-derivative in (\ref{fp}) is replaced by a more general
integro-differential operator with memory kernel; as a consequence, the
time-change is by means of the inverse of a general L\'{e}vy subordinator.

Another approach to fractional OU processes has been proposed in \cite{MAG3}
where the solution of the $\alpha $-stable Langevin equation (the so-called $%
\alpha $-stable OU process) is extended to the fractional case, tracing the
transition form Brownian motion to fractional Brownian motion. The resulting
processes exhibit long-range dependence, which, due to the lack of finite
second moments, must be measured by codifferences.

Other non-Gaussian OU-type processes have been obtained by applying the
Lamperti transform to the fractional Brownian motion (see \cite{CHE}) and
they are proved to display short-range dependence when the Hurst parameter $%
H $ is greater than $1/2$ (on the contrary to that obtained from the
Langevin equation with fractional noise). Also the extension of OU defined
in \cite{KAR}, by the Doob's transform of fractional Brownian motion, has
covariance function decaying exponentially for all values of $H.$

The class of time-fractional Pearson diffusions have also been studied in
\cite{LEO} and the fractional OU process is treated there as a particular
case. All these models are non-Markovian, as described in \cite{MUR}, and
non-Gaussian. Finally, we recall that fractional diffusion equations with
logarithmic-type differential operators have been studied in \cite{BEG} and,
as in the cases cited above, a random time change (or subordination) is
proved to hold for the corresponding process.

We present here different models, alternative to all those described above,
which generalize the OU process in a fractional (and non-fractional) sense,
but preserve the Gaussianity of its distribution. Moreover, contrary to the $%
\alpha $-stable OU and fractional $\alpha $-stable OU processes (in \cite%
{MAG3}), they have finite moments and, indeed, they are characterized by
means of their covariance function (since they are Gaussian).

As we will see, in our case, no random time change is entailed by the
introduction of the fractional operator, being in Fourier space. This is the
most important feature of our results and shows that correspondence between
partial differential equations with fractional (or convolution-type
derivatives) and subordination (or time-change) does not always come true.

The first model is called time-changed OU, since we prove that introducing,
in the Fourier transform of (\ref{fp}), a fractional operator entails a
deterministic time-change in the original process, expressed in terms of the
Mittag-Leffler function. Hence the process is still Gaussian, mean-reverting
and Markovian, even though, in the limit (for $t\rightarrow \infty ),$ it
does not converge to its stationary counterpart (as in the standard case),
but to a random walk. In order to correct this anomalous asymptotic
behavior, we define a second fractional OU process, whose (Gaussian)
one-dimensional density satisfies the same fractional FP equation, but is
not Markovian and has a stationary counterpart, in the limit. We show that
it presents long-range dependence, both in the covariance and in the
spectral density sense. We remark that, in all these cases, when the
fractional parameter tends to one, we recover the classical OU model.

As a motivation of the results described above, we show that the solution of
the fractional FP equation coincides with the one-dimensional distribution
of the limiting process obtained from a Poisson shot-noise, with power-law
decaying response function, when the frequency of the shocks increases and
their amplitude tends to zero.

Further extensions of the previous results are obtained, in the last
section, by generalizing the fractional operator adopted, by means of the
convolution-type derivative (see \cite{GAJ} and \cite{TOA} among the others)
and solving the corresponding FP equation (in Fourier space). The solution
coincides with the characteristic function of a generalized OU process,
which is then defined through the inverse of a general subordinator.

\

We now introduce some useful definitions and notation.

First of all, let $f:\mathbb{R}\rightarrow \mathbb{R}$ be an integrable
function belonging to the Schwartz space $\mathcal{S}(\mathbb{R})$ of
rapidly decreasing functions; then we denote by $\mathcal{F}(f;\xi
):=\int_{-\infty }^{+\infty }e^{i\xi x}f(x)dx=\widehat{f}(\xi )$ the Fourier
transform and by $\mathcal{F}^{-1}(f;x):=\frac{1}{2\pi }\int_{-\infty
}^{+\infty }e^{-i\xi x}\widehat{f}(\xi )d\xi $ its inverse. Moreover, for $%
s>0$, we denote the Laplace transform of $f$, as $\mathcal{L}%
(f;s):=\int_{0}^{+\infty }e^{-sx}f(x)dx=\widetilde{f}(s).$

Let $\alpha >0$, $m=\left\lfloor \alpha \right\rfloor +1$ and assume that $%
u:[a,b]\rightarrow \mathbb{R},$ $b>a,$\ is an absolutely continuous
function, with absolutely continuous derivatives up to order $m$ on $[a,b]$,
then, for $x\in \lbrack a,b],$
\begin{equation}
D_{x}^{\alpha }u(x):=\left\{
\begin{array}{l}
\frac{1}{\Gamma (m-\alpha )}\int_{a}^{x}\frac{1}{(x-s)^{\alpha -m+1}}\frac{%
d^{m}}{ds^{m}}u(s)ds,\qquad \alpha \notin \mathbb{N}_{0} \\
\frac{d^{m}}{dx^{m}},\qquad \alpha =m\in \mathbb{N}_{0}%
\end{array}%
\right.  \label{ca}
\end{equation}%
is the Caputo fractional derivative of order $\alpha $ (see \cite{KIL},
p.92). We define here the new following fractional operator.

\begin{definition}
\label{Def.1}
Let $\alpha \geq 0$ and $m=\left\lfloor \alpha \right\rfloor +1$. Let $%
u(\cdot ):[0,b],$ $b>0$, be a positive function, with absolutely continuous
derivatives up to order $m$ on $[0,b]$, then%
\begin{equation}
\mathcal{L}_{x}^{\alpha }u(x):=u(x)D_{x}^{\alpha }\log u(x),\qquad x\in
\mathbb{R},  \label{op}
\end{equation}%
where $D_{x}^{\alpha }$ is the Caputo fractional derivative of order $\alpha
.$
\end{definition}

It is easy to check that, for $\alpha =1,$ the operator defined in (\ref{op}%
) coincides with the first-order derivative, i.e.
\begin{equation}
\mathcal{L}_{x}^{1}u(x)=\frac{d}{dx}u(x).  \label{first}
\end{equation}

The previous definition can be generalized by considering the
convolution-type derivative (see \cite{TOA}), which is defined by means of
the subordinators' theory. Let $g:(0,+\infty )\rightarrow \mathbb{R}$ be a
Bernstein function, i.e. let $g$ be non-negative, infinitely differentiable
and such that, for any $x\in (0,+\infty ),$%
\begin{equation*}
(-1)^{n}\frac{d^{n}}{dx^{n}}g(x)\leq 0,\qquad \text{for any }n\in \mathbb{N}%
_{.}
\end{equation*}%
A function $g$ is a Bernstein function if and only if it admits the
following representation%
\begin{equation*}
g(x)=a+bx+\int_{0}^{+\infty }(1-e^{-sx})\overline{\nu }(ds),
\end{equation*}%
where $\overline{\nu }$ is the L\'{e}vy measure and $\left( a,b,\overline{%
\nu }\right) $ is called the L\'{e}vy triplet of $g$. Then a subordinator is
the stochastic process with non-decreasing paths $\mathcal{A}^{g}:=\left\{
\mathcal{A}^{g}(t),t\geq 0\right\} ,$ such that%
\begin{equation*}
\mathbb{E}e^{-s\mathcal{A}^{g}(t)}=e^{-g(s)t},
\end{equation*}%
(see \cite{APP} and \cite{SAT} for the theory of L\'{e}vy processes and
subordinators). Let moreover $L^{g}(t),$ $t\geq 0,$ be the inverse of $%
\mathcal{A}^{g}$, i.e.
\begin{equation*}
L^{g}(t)=\inf \left\{ s\geq 0:\mathcal{A}^{g}(s)>t\right\} ,\qquad t>0
\end{equation*}%
and $l_{g}(x,t)=\Pr \left\{ L^{g}(t)\in dx\right\} $ be its transition
density.

We recall the definition of the convolution-type derivative on the positive
half-axes, in the sense of Caputo (see \cite{TOA}, Def.2.4, for $b=0$) :%
\begin{equation}
D_{t}^{g}u(t):=\int_{0}^{t}\frac{d}{ds}u(t-s)\nu (s)ds,\qquad t>0,
\label{conv}
\end{equation}%
where $\nu $ is the tail of the L\'{e}vy measure $\overline{\nu },$ i.e. $%
\nu (s)ds=a+\overline{\nu }(s,\infty )ds$. For $g(s)=s^{\alpha }$ with $%
\alpha \in (0,1)$, the process $\mathcal{A}^{g}$ coincides with the $\alpha $%
-stable subordinator with $\mathbb{E}e^{-s\mathcal{A}^{g}(t)}=e^{-s^{\alpha
}t}$ (whose L\'{e}vy measure is $\overline{\nu }(s)=s^{-\alpha -1}/\Gamma
(m-\alpha )$) and formula (\ref{conv}) reduces to the Caputo derivative (\ref%
{ca}), for $m=1.$

The Laplace transform of $D_{t}^{g}$ is given by
\begin{equation}
\int_{0}^{+\infty }e^{-st}D_{t}^{g}u(t)dt=g(s)\widetilde{u}(s)-\frac{g(s)}{s}%
u(0),\qquad \mathcal{R}(s)>s_{0},  \label{lapconv}
\end{equation}%
(see \cite{TOA}, Lemma 2.5).

\

Let $E_{\beta ,\gamma }(x)$ denote the generalized Mittag-Leffler function
defined as%
\begin{equation}
E_{\beta ,\gamma }(x)=\sum_{j=0}^{\infty }\frac{x^{\beta j}}{\Gamma (\beta
j+\gamma )},\qquad x,\beta ,\gamma \in \mathbb{C},\text{ }\func{Re}\left(
\beta \right) ,\func{Re}\left( \gamma \right) >0,  \label{ml}
\end{equation}%
(see \cite{HAU}), then we recall the well-known formula of its Laplace
transform (see \cite{KIL}, formula (1.9.13), for $\rho =1$), i.e.%
\begin{equation}
\mathcal{L}\left\{ x^{\gamma -1}E_{\beta ,\gamma }(Ax^{\beta });s\right\} =%
\frac{s^{\beta -\gamma }}{s^{\beta }-A},  \label{lap}
\end{equation}%
with $\func{Re}\left( \beta \right) ,\func{Re}\left( \gamma \right) >0,$ $%
A\in \mathbb{R}$ and $s>|A|^{1/\func{Re}\left( \beta \right) }.$

We recall also the well-known (power-law) asymptotic behavior of the
Mittag-Leffler function (see \cite{KIL}, formula (1.8.11)), for $%
|z|\rightarrow +\infty $, i.e.%
\begin{equation}
E_{\beta ,\gamma }(z)=-\sum_{k=1}^{n}\frac{z^{-k}}{\Gamma (\gamma -\beta k)}%
+O\left( z^{-n-1}\right) ,\quad n\in \mathbb{N},  \label{asy}
\end{equation}%
which holds for $0<\beta <2,$ $\mu \leq \arg (z)\leq \pi $, where $\pi \beta
/2<\mu <\min \{\pi ,\pi \beta \}.$

We will also make use of the property of the function (\ref{ml}) of being
completely monotone (CM), which is proved in \cite{SCH}, for a negative
argument and $\beta \in (0,1]$, $\gamma \geq \beta .$ By definition, a
function $f(x)$ is CM if and only if it is infinitely differentiable and $%
\left( -1\right) ^{n}\frac{d^{n}}{dx^{n}}f(x)\geq 0$, for any $n=0,1,...$

Finally we recall the definition of Fox's H-functions, which will by applied
in the next sections. Let ${H}_{p,q}^{m,n}$ denote the H-function defined as
(see \cite{MAT} p.13):%
\begin{equation}
{H}_{p,q}^{m,n}\left[ \left. z\right\vert
\begin{array}{ccc}
(a_{1},A_{1}) & ... & (a_{p},A_{p}) \\
(b_{1},B_{1}) & ... & (b_{q},B_{q})%
\end{array}%
\right] =\frac{1}{2\pi i}\int_{\mathcal{L}}\frac{\left\{
\prod\limits_{j=1}^{m}\Gamma (b_{j}+B_{j}s)\right\} \left\{
\prod\limits_{j=1}^{n}\Gamma (1-a_{j}-A_{j}s)\right\} z^{-s}ds}{\left\{
\prod\limits_{j=m+1}^{q}\Gamma (1-b_{j}-B_{j}s)\right\} \left\{
\prod\limits_{j=n+1}^{p}\Gamma (a_{j}+A_{j}s)\right\} },  \label{hh}
\end{equation}%
with $z\neq 0,$ $m,n,p,q\in \mathbb{N}_{0},$ for $1\leq m\leq q$, $0\leq
n\leq p$, $a_{j},b_{j}\in \mathbb{R},$ $A_{j},B_{j}\in \mathbb{R}^{+},$ for $%
i=1,...,p,$ $j=1,...,q$ and $\mathcal{L}$ is a contour such that the
following condition is satisfied%
\begin{equation}
A_{\lambda }(b_{j}+\alpha )\neq B_{j}(a_{\lambda }-k-1),\qquad j=1,...,m,%
\text{ }\lambda =1,...,n,\text{ }\alpha ,k=0,1,...  \label{1.6}
\end{equation}

\section{Auxiliary results}

We consider the fractional diffusion equation obtained by replacing the time
derivative with the pseudo-differential operator $\mathcal{L}_{t}^{\alpha }$
defined in Def.\ref{Def.1} in the standard heat equation (in the Fourier space), i.e.%
\begin{equation*}
\frac{\mathcal{\partial }}{\partial t}\widehat{u}(\xi ,t)=-\frac{\xi ^{2}}{2}%
\widehat{u}(\xi ,t)\qquad \xi \in \mathbb{R},\text{ }t\geq 0,\qquad \text{%
with }\widehat{u}(\xi ,0)=1.
\end{equation*}

\begin{lemma}
\label{L2}
Let $\mathcal{L}_{t}^{\alpha }$ be the operator defined in Def.\ref{Def.1}, then the
solution to the following equation%
\begin{equation}
\mathcal{L}_{t}^{\alpha }\widehat{u}(\xi ,t)=-\frac{\Gamma (\alpha +1)\xi
^{2}}{2}\widehat{u}(\xi ,t)\qquad \xi \in \mathbb{R},\text{ }t\geq 0,\text{ }%
\alpha \in (0,1],  \label{diff}
\end{equation}%
is given by the characteristic function of a standard Brownian motion $%
W\doteq W(t),$ $t\geq 0,$\ time-changed by $\mathcal{T}(t)=t^{\alpha },$ for
any $t\geq 0,$ i.e. $W_{\alpha }(t):=W(t^{\alpha }).$
\end{lemma}

\begin{proof}
By considering (\ref{op}), we get from (\ref{diff}) that%
\begin{equation*}
D_{t}^{\alpha }\log \widehat{u}(\xi ,t)=-\frac{\Gamma (\alpha +1)\xi ^{2}}{2}%
,
\end{equation*}%
with $\log \widehat{u}(\xi ,0)=0.$ Then we take the Laplace transform and
obtain, by taking into account formula (2.4.62) in \cite{KIL} (with the
initial condition),%
\begin{equation}
s^{\alpha }\int_{0}^{+\infty }e^{-st}\log \widehat{u}(\xi ,t)dt-s^{\alpha
-1}\log \widehat{u}(\xi ,0)=-\frac{\Gamma (\alpha +1)\xi ^{2}}{2s}
\end{equation}%
and then%
\begin{equation*}
\log \widehat{u}(\xi ,t)=-\frac{\xi ^{2}t^{\alpha }}{2},
\end{equation*}%
which gives
\begin{equation*}
\widehat{u}(\xi ,t)=e^{-\xi ^{2}t^{\alpha }/2},\qquad \xi \in \mathbb{R},%
\text{ }t\geq 0.
\end{equation*}
\end{proof}

We remark that the introduction of the operator $\mathcal{L}_{t}^{\alpha }$
entails a deterministic time change in the corresponding process, contrary
to what happens when substituting the time-derivative with the Caputo (or
the Riemann-Liouville) fractional derivative in the heat equation. In the
latter case, a random time-change is produced, by means of an independent
inverse stable subordinator (see, e.g., \cite{MEE}, p.42).

In this case, since the function $\mathcal{T}_{\alpha }(\cdot ):\mathbb{R}%
^{+}\rightarrow \mathbb{R}^{+}$ is continuous, increasing and such that $%
\mathcal{T}_{\alpha }(0)=0$, as a consequence, $W_{\alpha }(t)$ is an
additive process (see \cite{CON}, p. 469 and \cite{SAT}, p.3), i.e. the
Brownian motion (which is a L\'{e}vy process) looses the stationarity of its
increments, by effect of the time-change, but the latter are still
independent. Indeed, we get, for $0\leq t_{1}\leq t_{2},$%
\begin{equation*}
\mathbb{E}\left( W_{\alpha }(t_{2})-W_{\alpha }(t_{1})\right)
^{2}=t_{2}^{\alpha }-t_{1}^{\alpha }\neq \left( t_{2}-t_{1}\right) ^{\alpha
}=\mathbb{E}\left( W_{\alpha }(t_{2}-t_{1})\right) ^{2},
\end{equation*}%
unless $\alpha =1.$ Finally, we notice that $W_{\alpha }$ is a self-similar
process with Hurst exponent equal to $\alpha /2,$ since, for any $c>0,$
\begin{equation*}
c^{\alpha /2}W_{\alpha }(t)\overset{d}{=}W_{\alpha }(ct),\qquad \text{for
any }t\geq 0.
\end{equation*}%
The process $W_{\alpha }$ is already known in the literature as "scaled
Brownian motion" (see \cite{LIM}, \cite{JEO}), and is used to model
anomalous diffusion (i.e. diffusion with mean-squared displacement $\langle
W_{\alpha }(t)^{2}\rangle \sim t^{\alpha }$, for $\alpha \in (0,2)$). As we
will see in the next section, the introduction of the operator $\mathcal{L}%
_{t}^{\alpha }$ in the Fokker-Planck equation (in Fourier space) will entail
a deterministic time-change in the corresponding OU process, which, instead
of $\mathcal{T}_{\alpha }(t)=t^{\alpha },$ is expressed by means of the
Mittag-Leffler process (see formula (\ref{tt}) below).

\

As a motivation for our analysis, we introduce the following Poisson
shot-noise process and study its asymptotic behavior.

\begin{definition}
\label{Def.3}
\textbf{(Power-law shot-noise) }Let $\left\{ N(t),t\geq 0\right\} $ be an
homogeneous Poisson process, with rate $\lambda >0,$ and let $T_{j}$ be the
occurring time of its $j$-th event, for $j=1,2,...,$ then we define the
generalized shot-noise process as%
\begin{equation}
S_{\mathcal{\alpha }}(t):=\sum_{j=1}^{N(t)}h(t-T_{j}),  \label{sh}
\end{equation}%
with $h(u)=\sqrt{u^{\mathcal{\alpha }-1}E_{\alpha ,\mathcal{\alpha }%
}(-\gamma (2u)^{\mathcal{\alpha }})}1_{\left\{ u>0\right\} }$, $\alpha \in
(0,1].$
\end{definition}

\begin{remark}
By considering formula (16.85) in \cite{PAP}, p.632, for $h(\cdot )$ given
above, it can be easily checked that the.joint cumulant generating function
of $S_{\mathcal{\alpha }}$ is equal to%
\begin{equation}
C_{\mathcal{\alpha }}(\underline{\theta },\underline{t}):=\log \mathbb{E}%
e^{\theta _{1}S_{\mathcal{\alpha }}(t_{1})+...+\theta _{k}S_{\mathcal{\alpha
}}(t_{k})}=\lambda \int_{-\infty }^{+\infty }\left[ e^{\theta
_{1}h(t_{1}-u)+...+\theta _{k}h(t_{k}-u)}-1\right] du,  \label{jgf}
\end{equation}%
where $\underline{\theta }:=(\theta _{1},...,\theta _{k})$ and $\underline{t}%
:=(t_{1},..,t_{k}),$ for any $k\in \mathbb{N}.$ Moreover, by applying
formula (\ref{asy}), we can see that $h(u)\simeq ku^{-\frac{\mathcal{\alpha }%
}{2}-\frac{1}{2}},$ for $u\rightarrow \infty $ and for a constant $k>0.$ The
function $h$ is decreasing, since both terms under square root in its
expression are CM, for $\mathcal{\alpha }\in (0,1]$, and thus the same is
true for their product; the square root does not affect the decreasing
behavior. Therefore the previous model is well suited to the case where the
values of the shocks (occurring at Poisson times) decrease in time with a
power law (thus slower than in the, usual, exponential case).
\end{remark}

We now study the limiting behavior of $S_{\mathcal{\alpha }}:=\{S_{\mathcal{%
\alpha }}(t),t\geq 0\}$, after centering (\ref{sh}) (subtracting its mean $%
\mu _{n}$) and by a proper rescaling of the amplitude of the shocks and
their frequency.

\begin{lemma}
Let $\left\{ N_{n}(t),t\geq 0\right\} $ be an homogeneous Poisson process,
with rate $\lambda =n\lambda _{0},$ for $\lambda _{0}>0,$ and let $h(u)=%
\frac{1}{\sqrt{n}}h_{0}(u),$ with $h_{0}(u)=\sqrt{(u+\xi _{0})^{\mathcal{%
\alpha }-1}E_{\alpha ,\mathcal{\alpha }}(-\gamma 2^{\alpha }(u+\xi
_{0})^{\alpha })}1_{\left\{ u>0\right\} },$ for $n\in \mathbb{N}$, some $\xi
_{0}>0.$ Under these assumptions, the process%
\begin{equation}
U_{\mathcal{\alpha },n}(t):=\frac{1}{\sqrt{n}}\left[
\sum_{j=1}^{N_{n}(t)}h_{0}(t-T_{j})-\mu _{n}(t)\right] ,  \label{uu}
\end{equation}%
converges weakly, for $n\rightarrow +\infty $, to a Gaussian process $U_{%
\mathcal{\alpha }}:=\{U_{\mathcal{\alpha }}(t),t\geq 0\}$ with joint
cumulant generating function%
\begin{equation}
\log \mathbb{E}e^{\theta _{1}U_{\mathcal{\alpha }}(t_{1})+...+\theta _{k}U_{%
\mathcal{\alpha }}(t_{k})}=\frac{\lambda _{0}}{2}\int_{-\infty }^{+\infty }%
\left[ \theta _{1}h_{0}(t_{1}-u)+...+\theta _{k}h_{0}(t_{k}-u)\right] ^{2}du.
\label{jgf2}
\end{equation}
\end{lemma}

\begin{proof}
The convergence of the finite dimensional distributions can be easily
derived, following \cite{PAP}, p.633, by the Taylor's series expansion of (%
\ref{jgf}). See also \cite{PAP2} and \cite{HEI}. Then we prove tightness of
the sequence, by proving the following sufficient conditions (see \cite{REV}%
, p.516):

\begin{equation*}
P(|U_{\mathcal{\alpha },n}(0)|>A)\leq \varepsilon
\end{equation*}%
for any $n\geq n_{0}$, $\varepsilon >0$ and for some $n_{0},A$,%
\begin{equation}
\delta ^{-1}P(\sup_{|t-s|<\delta }|U_{\mathcal{\alpha },n}(t)-U_{\mathcal{%
\alpha },n}(s)|\geq \eta )\leq \varepsilon  \label{con}
\end{equation}%
for $s<t$ and for some $\delta \in (0,1).$ The first condition is trivially
satisfied since the process starts a.s. from zero. The check that condition (%
\ref{con}) holds is carried out by applying the Doob submartingale's
inequality (see \cite{REV}, p.52), i.e.%
\begin{equation*}
\eta ^{p}P(\sup_{t}X_{t}\geq \eta )\leq \sup_{t}\mathbb{E}X_{t}^{p},
\end{equation*}%
for $p\geq 1$ and for a positive submartingale $X_{t}.$ We choose $p=4$ and
show that the difference process in (\ref{con}) is a positive submartingale,
since its fourth moment is increasing in $t$, for fixed $t-s$. Indeed, we
can write (by denoting $\overline{h}_{0}(\cdot )$ the centered addends)%
\begin{eqnarray*}
U_{\mathcal{\alpha },n}(t)-U_{\mathcal{\alpha },n}(s) &=&\frac{1}{\sqrt{n}}%
\left[ \sum_{j=1}^{N_{n}(t)}h_{0}(t-T_{j})-\mu
_{n}(t)-\sum_{j=1}^{N_{n}(s)}h_{0}(s-T_{j})+\mu _{n}(s)\right] \\
&=&\frac{1}{\sqrt{n}}\left[ \sum_{j=1}^{N_{n}(s)}\overline{h}%
_{0}(t-T_{j})+\sum_{j=N_{n}(s)+1}^{N_{n}(t)}\overline{h}_{0}(t-T_{j})-%
\sum_{j=1}^{N_{n}(s)}\overline{h}_{0}(s-T_{j})\right] \\
&=&\frac{1}{\sqrt{n}}\left[ \sum_{j=N_{n}(s)+1}^{N_{n}(t)}\overline{h}%
_{0}(t-T_{j})-\sum_{j=1}^{N_{n}(s)}\left[ \overline{h}_{0}(s-T_{j})-%
\overline{h}_{0}(t-T_{j})\right] \right] \\
&=&U_{\alpha ,n}^{\ast }(t-s)-W_{\alpha ,n}(s;t)
\end{eqnarray*}%
where $U_{\alpha ,n}^{\ast }$ is a new, centered, shot-noise (starting from $%
s$ and evaluated in $t$), which depends only on the Poisson events in $%
\left( 0,t-s\right) $ and is independent of
\begin{equation*}
W_{\alpha ,n}(s;t):=\sum_{j=1}^{N_{n}(s)}\left[ \overline{h}_{0}(s-T_{j})-%
\overline{h}_{0}(t-T_{j})\right] /\sqrt{n},
\end{equation*}%
thanks to the lack of memory property of the Poisson process. Thus we get,
by putting $\rho :=t-s,$%
\begin{eqnarray}
&&\mathbb{E}\left[ U_{\mathcal{\alpha },n}(t)-U_{\mathcal{\alpha },n}(s)%
\right] ^{4}  \label{uw} \\
&=&\mathbb{E}\left[ U_{\alpha ,n}^{\ast }(\rho )\right] ^{4}+\mathbb{E}\left[
W_{\alpha ,n}(s;\rho )\right] ^{4}-4\mathbb{E}\left[ U_{\alpha ,n}^{\ast
}(\rho )\right] \mathbb{E}\left[ W_{\alpha ,n}(s;\rho )\right] ^{3}+  \notag
\\
&&-4\mathbb{E}\left[ U_{\alpha ,n}^{\ast }(\rho )\right] ^{3}\mathbb{E}\left[
W_{\alpha ,n}(s;\rho )\right] +6\mathbb{E}\left[ U_{\alpha ,n}^{\ast }(\rho )%
\right] ^{2}\mathbb{E}\left[ W_{\alpha ,n}(s;\rho )\right] ^{2}  \notag \\
&=&\mathbb{E}\left[ U_{\alpha ,n}^{\ast }(\rho )\right] ^{4}+\mathbb{E}\left[
W_{\alpha ,n}(s;\rho )\right] ^{4}+6\mathbb{E}\left[ U_{\alpha ,n}^{\ast
}(\rho )\right] ^{2}\mathbb{E}\left[ W_{\alpha ,n}(s;\rho )\right] ^{2}
\notag
\end{eqnarray}%
since, by the centering, the first moment of both $U_{\alpha ,n}^{\ast }$ \
and $W_{\alpha ,n}(s;\rho )$ is null. We can now evaluate the moments of $%
W_{\alpha ,n}(s;\rho ),$ by its cumulant generating function:%
\begin{eqnarray}
K(\theta ) &:&=\log \{\mathbb{E}e^{\theta W_{\alpha ,n}(s;\rho )}\}=\log
\left\{ e^{-\lambda s}\sum_{m=0}^{\infty }\frac{(\lambda s)^{m}}{m!}\mathbb{E%
}\left[ \left. e^{\theta W_{\alpha ,n}(s;\rho )}\right\vert N(s)=m\right]
\right\}  \label{cum} \\
&=&\log \left\{ e^{-\lambda s}\sum_{m=0}^{\infty }\frac{(\lambda s)^{m}}{m!}%
\left[ \frac{1}{s}\int_{0}^{s}e^{\frac{\theta }{\sqrt{n}}\left[ \overline{h}%
_{0}(s-y)-\overline{h}_{0}(s+\rho -y)\right] }dy\right] ^{m}\right\}  \notag
\\
&=&\lambda \int_{0}^{s}\left[ e^{\frac{\theta }{\sqrt{n}}\left[ \overline{h}%
_{0}(s-y)-\overline{h}_{0}(s+\rho -y)\right] }-1\right] dy,  \notag
\end{eqnarray}%
where, in the second line, we have considered the uniform conditional
distribution of the occurring time for each Poisson event and their
independence. Then, by denoting as $\kappa _{j}$ the $j$-th cumulant of $%
W_{\alpha ,n}(s;\rho )$ and differentiating (\ref{cum}), we get
\begin{equation}
\mathbb{E}\left[ W_{\alpha ,n}(s;\rho )\right] ^{2}=\kappa _{2}-\kappa
_{1}^{2}=\left. \frac{d^{2}}{d\theta ^{2}}K(\theta )\right\vert _{\theta =0}=%
\frac{\lambda }{\sqrt{n}}\int_{0}^{s}\left[ \overline{h}_{0}(s-y)-\overline{h%
}_{0}(s+\rho -y)\right] ^{2}dy  \label{uw2}
\end{equation}%
and%
\begin{eqnarray}
&&\mathbb{E}\left[ W_{\alpha ,n}(s;\rho )\right] ^{4}  \label{uw3} \\
&=&\kappa _{4}+4\kappa _{3}\kappa _{1}+3\kappa _{2}^{2}+6\kappa _{2}\kappa
_{1}^{2}+\kappa _{1}^{4}=\left. \frac{d^{4}}{d\theta ^{4}}K(\theta
)\right\vert _{\theta =0}+3\left[ \left. \frac{d^{2}}{d\theta ^{2}}K(\theta
)\right\vert _{\theta =0}\right] ^{2}  \notag \\
&=&\frac{\lambda }{n^{2}}\int_{0}^{s}\left[ \overline{h}_{0}(s-y)-\overline{h%
}_{0}(s+\rho -y)\right] ^{4}dy+\frac{3\lambda ^{2}}{n}\left\{ \int_{0}^{s}%
\left[ \overline{h}_{0}(s-y)-\overline{h}_{0}(s+\rho -y)\right]
^{2}dy\right\} ^{2},  \notag
\end{eqnarray}%
(by considering that the first cumulant vanishes). We now take the sup of (%
\ref{uw}), for $\rho \leq \delta $:%
\begin{eqnarray*}
&&\sup_{\rho \leq \delta }\mathbb{E}\left[ U_{\mathcal{\alpha },n}(s+\rho
)-U_{\mathcal{\alpha },n}(s)\right] ^{4} \\
&\leq &\sup_{\rho \leq \delta }\mathbb{E}\left[ U_{\alpha ,n}^{\ast }(\rho )%
\right] ^{4}+\sup_{\rho \leq \delta }\mathbb{E}\left[ W_{\alpha ,n}(s;\rho )%
\right] ^{4}+6\sup_{\rho \leq \delta }\mathbb{E}\left[ U_{\alpha ,n}^{\ast
}(\rho )\right] ^{2}\sup_{\rho \leq \delta }\mathbb{E}\left[ W_{\alpha
,n}(s;\rho )\right] ^{2}.
\end{eqnarray*}%
All the terms are positive and increasing in $\rho $, for fixed $s$, as can
be checked in (\ref{uw2}) and (\ref{uw3}), by considering also that%
\begin{equation*}
U_{\alpha ,n}^{\ast }(\rho )=\frac{1}{\sqrt{n}}\sum_{j=N_{n}(s)+1}^{N_{n}(s+%
\rho )}\overline{h}_{0}(s+\rho -T_{j})\leq \frac{1}{\sqrt{n}}%
\sum_{j=N_{n}(s)+1}^{N_{n}(s+\rho )}\overline{h}_{0}(s-T_{j})
\end{equation*}%
and that $\overline{h}_{0}(\cdot )$ is positive. Thus we get that%
\begin{equation}
\sup_{\rho \leq \delta }\mathbb{E}\left[ U_{\mathcal{\alpha },n}(s+\rho )-U_{%
\mathcal{\alpha },n}(s)\right] ^{4}\leq \mathbb{E}\left[ U_{\alpha ,n}^{\ast
}(\delta )\right] ^{4}+\mathbb{E}\left[ W_{\alpha ,n}(s;\delta )\right]
^{4}+6\mathbb{E}\left[ U_{\alpha ,n}^{\ast }(\delta )\right] ^{2}\mathbb{E}%
\left[ W_{\alpha ,n}(s;\delta )\right] ^{2}.  \label{of}
\end{equation}%
In order to study the limit for $s\rightarrow \infty $, of (\ref{of}), we
note that, for $x,\rho >0,$%
\begin{equation}
\left\vert \overline{h}_{0}(x)-\overline{h}_{0}(x+\rho )\right\vert \leq
\rho \sup_{x<z<x+\rho }\frac{d}{dz}\overline{h}_{0}(z)\leq \frac{k\rho }{%
(x+\xi _{0})^{\beta }},  \label{beta}
\end{equation}%
for some $k>0$ and $\beta >\frac{3}{2}.$ Indeed, by ignoring the expected
value, we have that
\begin{eqnarray*}
\frac{d}{dx}\overline{h}_{0}(x) &=&\frac{d}{dx}\left[ \sqrt{(x+\xi _{0})^{%
\mathcal{\alpha }-1}E_{\alpha ,\mathcal{\alpha }}(-\gamma (2x+\xi _{0})^{%
\mathcal{\alpha }})}\right] \\
&=&\frac{\sum_{j=0}^{\infty }\frac{(-\gamma 2^{\mathcal{\alpha }})^{j}(x+\xi
_{0})^{\alpha j+\alpha -2}}{\Gamma \left( \alpha j+\alpha -1\right) }}{2%
\sqrt{(x+\xi _{0})^{\mathcal{\alpha }-1}E_{\alpha ,\mathcal{\alpha }%
}(-\gamma 2^{\alpha }(x+\xi _{0})^{\mathcal{\alpha }})}}=\frac{x^{\frac{%
\alpha }{2}-\frac{3}{2}}E_{\alpha ,\mathcal{\alpha -}1}(-\gamma 2^{\alpha
}(x+\xi _{0})^{\mathcal{\alpha }})}{2\sqrt{E_{\alpha ,\mathcal{\alpha }%
}(-\gamma 2^{\alpha }(x+\xi _{0})^{\mathcal{\alpha }})}}.
\end{eqnarray*}%
Whence, using (\ref{asy})$,$
\begin{eqnarray}
\sup_{x<z<x+\rho }\frac{d}{dz}\overline{h}_{0}(z) &\leq &\sup_{z\geq x}\frac{%
d}{dz}\overline{h}_{0}(z)  \label{sup} \\
&\leq &\sup_{z\geq x}\frac{\left[ K_{1}(z+\xi _{0})^{-2\alpha
}+O(z^{-3\alpha })\right] /\Gamma (-\alpha -1)}{(z+\xi _{0})^{\frac{3}{2}-%
\frac{\alpha }{2}}\left[ K_{2}(z+\xi _{0})^{-2\alpha }+O(z^{-3\alpha })%
\right] ^{1/2}/\Gamma (-\alpha )}  \notag \\
&\leq &\frac{K_{3}}{(z+\xi _{0})^{\beta }}  \notag
\end{eqnarray}%
(for some positive constants $K_{1},$ $K_{2}$ $K_{3},$ so that (\ref{beta})
is satisfied with $\beta =\frac{\alpha }{2}+\frac{3}{2}.$ As a consequence,%
\begin{eqnarray*}
&&\int_{0}^{s}\left[ \overline{h}_{0}(s-y)-\overline{h}_{0}(s+\rho -y)\right]
^{2}dy \\
&=&\int_{0}^{s-c}\left[ \overline{h}_{0}(s-y)-\overline{h}_{0}(s+\rho -y)%
\right] ^{2}dy+\int_{s-c}^{s}\left[ \overline{h}_{0}(s-y)-\overline{h}%
_{0}(s+\rho -y)\right] ^{2}dy.
\end{eqnarray*}%
The second integral can be bounded by considering that, for any $c>0$ and $%
u\in \lbrack 0,c]$,%
\begin{eqnarray*}
|\overline{h}_{0}(u)-\overline{h}_{0}(u+\rho )| &\leq &\rho \sup_{0\leq
z\leq c}\frac{d}{dz}\overline{h}_{0}(z) \\
&=&[\text{analogously to (\ref{sup})}] \\
&\leq &\frac{K_{3}\rho }{\xi _{0}^{\beta }}.
\end{eqnarray*}
\end{proof}

The first moments of the limiting process $U_{\mathcal{\alpha }}$ can be
evaluated by differentiating (\ref{jgf2}) and reads: $\mathbb{E}U_{\mathcal{%
\alpha }}(t)=0$ and
\begin{eqnarray}
Var(U_{\mathcal{\alpha }}(t)) &=&\lambda _{0}\int_{-\infty }^{+\infty
}h_{0}^{2}(t-u)du=\lambda _{0}\int_{0}^{t}(u+\xi _{0})^{\mathcal{\alpha }%
-1}E_{\alpha ,\mathcal{\alpha }}(-\gamma 2^{\mathcal{\alpha }}(u+\xi _{0})^{%
\mathcal{\alpha }})du  \label{var} \\
&=&-\lambda _{0}\left. E_{\mathcal{\alpha },1}\left( -\gamma 2^{\mathcal{%
\alpha }}(u+\xi _{0})^{\mathcal{\alpha }}\right) \right\vert
_{u=0}^{t}=\lambda _{0}\left[ E_{\mathcal{\alpha },1}(-\gamma 2^{\mathcal{%
\alpha }}\xi _{0}^{\alpha })-E_{\mathcal{\alpha },1}(-\gamma 2^{\mathcal{%
\alpha }}(t+\xi _{0})^{\mathcal{\alpha }})\right] ,  \notag
\end{eqnarray}%
by considering that%
\begin{equation*}
\frac{d}{dw}E_{\mathcal{\alpha },1}(-kw^{\mathcal{\alpha }})=-kw^{\mathcal{%
\alpha }-1}E_{\mathcal{\alpha },\mathcal{\alpha }}(-kw^{\mathcal{\alpha }%
}),\qquad k,w\in \mathbb{R}\text{.}
\end{equation*}%
The previous Lemma shows that the distribution of a suitably normalized
sequence of Poisson shot-noise (with power-law impulse response function
described by $h_{0}$) converges, when the frequency of shocks is large and
their amplitude small, to a generalized version of the OU process.\ Indeed,
for $\xi _{0}=0,$ one obtains the following one-dimensional distribution
\begin{eqnarray}
f_{U_{\mathcal{\alpha }}(t)}(x) &:&=P\{U_{\mathcal{\alpha }}(t)\in dx\}/dx
\label{uua} \\
&=&\frac{1}{\sqrt{2\pi (\gamma /\theta )\left[ 1-E_{\alpha ,1}(-\gamma
(2t)^{\alpha })\right] }}\exp \left\{ -\frac{x^{2}}{2(\gamma /\theta )\left[
1-E_{\alpha ,1}(-\gamma (2t)^{\alpha })\right] }\right\} ,  \notag
\end{eqnarray}%
for $x\in \mathbb{R},t>0,\gamma ,\theta >0.$ Note that, as far as the
convergence of the finite-dimensional distributions is concerned, $\xi _{0}$
could be allowed to take any fixed non-negative value $\xi _{0}\in \mathbb{R}%
_{0}^{+}$; however the proof of tightness requires the slightly stronger
condition $\xi _{0}>0.$

\section{Main results}

\subsection{Time-changed Ornstein-Uhlenbeck process}

We start by evaluating the solution of a fractional version of the
Fokker-Planck equation (in the Fourier space), defined by means of the
operator introduced above.

\begin{theorem}
\label{Thm6}
Let $\mathcal{L}_{t}^{\alpha }$ be the operator defined in Def.\ref{Def.1}, then the
solution of the following equation%
\begin{equation}
\mathcal{L}_{t}^{\alpha }\widehat{u}(\xi ,t)=-\frac{\gamma }{2^{1-\alpha }}%
\xi \frac{\partial }{\partial \xi }\widehat{u}(\xi ,t)-\frac{\theta }{%
2^{1-\alpha }}\xi ^{2}\widehat{u}(\xi ,t),\qquad \xi \in \mathbb{R},\text{ }%
t\geq 0,\text{ }\alpha \in (0,1],\text{ }D,\gamma >0,  \label{ch}
\end{equation}%
with initial condition $\widehat{u}(\xi ,0)=1,$ is given by
\begin{equation}
\widehat{u}(\xi ,t)=\exp \left\{ -\frac{\mathcal{\theta }\xi ^{2}}{2\gamma }%
\left[ 1-E_{\alpha ,1}(-\gamma (2t)^{\alpha })\right] \right\} .  \label{ch2}
\end{equation}
\end{theorem}

\begin{proof}
We start by writing the Fourier transform of the FP equation, given in (\ref%
{fp}), i.e.%
\begin{eqnarray}
&&\frac{\partial }{\partial t}\widehat{u}(\xi ,t)=\eta \int_{-\infty
}^{+\infty }e^{i\xi x}\frac{\partial }{\partial x}[x u(x,t)]dx-D\xi ^{2}%
\widehat{u}(\xi ,t)  \notag \\
&=&\eta \left\{ \left[ e^{i\xi x}x u(x,t)\right] _{-\infty }^{+\infty }-i\xi
\int_{-\infty }^{+\infty }e^{i\xi x}x u(x,t)dx\right\} -D\xi ^{2}\widehat{u}%
(\xi ,t),  \notag
\end{eqnarray}%
which, with $\eta =\gamma /2^{1-\alpha }$ and $D=\theta /2^{1-\alpha }$,
coincides with the r.h.s. of (\ref{ch}), by considering the boundary
condition.

Since the Mittag-Leffler function (\ref{ml}) is the eigenfunction of the
Caputo fractional derivative (\ref{ca}) (see Lemma 2.23, p.98, in \cite{KIL}%
), we get from (\ref{ch2}), that%
\begin{eqnarray*}
D_{t}^{\alpha }\log \widehat{u}(\xi ,t) &=&-\frac{\mathcal{\theta }\xi ^{2}}{%
2\gamma }D_{t}^{\alpha }\left[ 1-E_{\alpha ,1}(-\gamma (2t)^{\alpha })\right]
\\
&=&-\frac{\mathcal{\theta }\xi ^{2}}{2^{1-\alpha }}E_{\alpha ,1}(-\gamma
(2t)^{\alpha }) \\
&=&-2^{\alpha }\gamma \log \widehat{u}(\xi ,t)-\frac{\mathcal{\theta }\xi
^{2}}{2^{1-\alpha }}
\end{eqnarray*}%
and thus%
\begin{equation*}
\widehat{u}(\xi ,t)D_{t}^{\alpha }\log \widehat{u}(\xi ,t)=\frac{\mathcal{%
\theta }\xi ^{2}}{2^{1-\alpha }}\left[ 1-E_{\alpha ,1}(-\gamma (2t)^{\alpha
})\right] \widehat{u}(\xi ,t)-\frac{\mathcal{\theta }\xi ^{2}}{2^{1-\alpha }}%
\widehat{u}(\xi ,t),
\end{equation*}%
which coincides with (\ref{ch}). The initial condition is verified since $%
E_{\alpha ,1}(0)=1.$
\end{proof}

\begin{remark}
Formula (\ref{ch2}) coincides with the characteristic function of the
process $U_{\alpha }(t)$ obtained, in the previous section, by taking the
limit of the power-law Poisson shot-noise defined in Def.\ref{Def.3}. Moreover, for $%
\alpha =1,$ it reduces to the characteristic function of the OU process
(with starting point in the origin), i.e.%
\begin{equation*}
\widehat{u}(\xi ,t)=\exp \left\{ -\frac{\mathcal{\theta }\xi ^{2}}{2\gamma }%
\left[ 1-e^{-2\gamma t}\right] \right\} .
\end{equation*}
\end{remark}

Similarly to the result obtained in Lemma \ref{L2}, the solution of (\ref{ch})
coincides with the characteristic function of the OU process time-changed by
means of the continuous and increasing function
\begin{equation}
\mathcal{T}_{\alpha }(t)=-\frac{1}{2\gamma }\log E_{\alpha ,1}(-\gamma
(2t)^{\alpha }),  \label{tt}
\end{equation}%
such that $\mathcal{T}_{\alpha }(0)=0.$ Therefore we define the time-changed
OU process, by means of (\ref{tt}), as follows.

\begin{definition}
\textbf{(Time-changed OU process) }Let $X_{1}:=\left\{ X_{1}(t),t\geq
0\right\} $ denote the standard OU process, then we define the time-changed
OU process $\left\{ X_{\alpha }=X_{\alpha }(t),t\geq 0\right\} $, as $%
X_{\alpha }(t):=X_{1}(\mathcal{T}_{\alpha }(t)),$ for $\mathcal{T}_{\alpha
}(\cdot )$ given in (\ref{tt}).

\begin{remark}
The process $X_{\alpha }$ is still Gaussian (being obtained by a
deterministic time-change) with $\mathbb{E}X_{\alpha }(t)=0$, and, for any $%
t,s\geq 0$, $\gamma >0,$ $\alpha \in (0,1],$ has%
\begin{equation}
Cov\left( X_{\alpha }(t),X_{\alpha }(s)\right) =\frac{\mathcal{\theta }}{%
\gamma }\sqrt{\frac{E_{\alpha ,1}(-\gamma 2^{\alpha }(t\vee s)^{\alpha })}{%
E_{\alpha ,1}(-\gamma 2^{\alpha }(t\wedge s)^{\alpha })}}\left[ 1-E_{\alpha
,1}(-\gamma 2^{\alpha }(t\wedge s)^{\alpha })\right] .  \label{cov}
\end{equation}%
Formula (\ref{cov}) can be derived as follows%
\begin{eqnarray*}
Cov\left( X_{\alpha }(t),X_{\alpha }(s)\right) &=&\frac{\mathcal{\theta }}{%
\gamma }e^{-\gamma \mathcal{T}_{\alpha }(t)|}\left[ e^{\gamma \mathcal{T}%
_{\alpha }(s)}-e^{-\gamma \mathcal{T}_{\alpha }(s)}\right] \\
&=&\frac{\mathcal{\theta }}{\gamma }e^{\frac{1}{2}\log E_{\alpha ,1}(-\gamma
(2t)^{\alpha })}\left[ e^{-\frac{1}{2}\log E_{\alpha ,1}(-\gamma
(2s)^{\alpha })}-e^{\frac{1}{2}\log E_{\alpha ,1}(-\gamma (2s)^{\alpha })}%
\right]
\end{eqnarray*}%
for $s<t$ and thus for $\mathcal{T}_{\alpha }(s)<\mathcal{T}_{\alpha }(t)$,
by recalling that the Mittag-Leffler function $E_{\alpha ,1}(x)$ is
completely monotone, for $\alpha \in (0,1),$ on the real negative semi-axis.
Analogously, for $s\geq t.$
\end{remark}
\end{definition}

\begin{remark}
We can easily derive the following representation of $X_{\alpha }$ in terms
of a Brownian motion, by applying again the time change (\ref{tt}) to the
well known representation of the standard OU process:
\begin{eqnarray}
X_{\alpha }(t) &=&\sqrt{\frac{\mathcal{\theta }}{\gamma }}e^{-\gamma
\mathcal{T}_{\alpha }(t)}W(e^{2\gamma \mathcal{T}_{\alpha }(t)}-1)
\label{xa} \\
&=&\sqrt{\frac{\mathcal{\theta }}{\gamma }E_{\alpha ,1}(-2^{\alpha }\gamma
t^{\alpha })}W\left( \frac{1}{E_{\alpha ,1}(-2^{\alpha }\gamma t^{\alpha })}%
-1\right) .  \notag
\end{eqnarray}%
Moreover, it follows from (\ref{cov}) that the process is Markovian, for any
$\alpha $, since
\begin{equation*}
\rho _{X_{\alpha }}(s,t)=\sqrt{\frac{E_{\alpha ,1}(-\gamma 2^{\alpha
}t^{\alpha })\left[ 1-E_{\alpha ,1}(-\gamma 2^{\alpha }s^{\alpha })\right] }{%
E_{\alpha ,1}(-\gamma 2^{\alpha }s^{\alpha })\left[ 1-E_{\alpha ,1}(-\gamma
2^{\alpha }t^{\alpha })\right] }}=\rho _{X_{\alpha }}(s,h)\rho _{X_{\alpha
}}(h,t),\qquad s<h<t,
\end{equation*}%
where $\rho _{X_{\alpha }}(\cdot ,\cdot )$ denotes the autocorrelation
coefficient. On the other hand, in this case, by considering formula (\ref%
{asy})$,$ we get, for any $\tau \geq 0,$%
\begin{equation*}
\lim_{t\rightarrow \infty }Cov\left( X_{\alpha }(t),X_{\alpha }(t+\tau
)\right) =\frac{\mathcal{\theta }}{\gamma }\lim_{t\rightarrow \infty }\sqrt{%
\frac{E_{\alpha ,1}(-\gamma 2^{\alpha }(t+\tau )^{\alpha })}{E_{\alpha
,1}(-\gamma 2^{\alpha }t^{\alpha })}}\left[ 1-E_{\alpha ,1}(-\gamma
2^{\alpha }t^{\alpha })\right] =\frac{\mathcal{\theta }}{\gamma }.
\end{equation*}%
Thus the process $X_{\alpha }$ asymptotically behaves as a random walk.
Moreover it does not coincide, in the limit, with the stationary OU process
(which we denote by $\overline{X}_{1}$) time-changed by (\ref{tt}), i.e.
with $\overline{X}_{\alpha }(t)=\overline{X}_{1}(\mathcal{T}_{\alpha }(t)),$
for $\mathcal{T}_{\alpha }(\cdot )$ given in (\ref{tt}). Indeed the latter
is a Gaussian (non-stationary) process with $\mathbb{E}\overline{X}_{\alpha
}(t)=0$, and%
\begin{equation}
Cov\left( \overline{X}_{\alpha }(t),\overline{X}_{\alpha }(s)\right) =\frac{%
\mathcal{\theta }}{\gamma }\sqrt{\frac{E_{\alpha ,1}(-\gamma 2^{\alpha
}(t\vee s)^{\alpha })}{E_{\alpha ,1}(-\gamma 2^{\alpha }(t\wedge s)^{\alpha
})}},  \label{ss}
\end{equation}%
for any $t,s\geq 0$, $\gamma >0,$ $\alpha \in (0,1].$ Formula (\ref{ss}) is
obtained as follows: for $t>s$%
\begin{eqnarray*}
Cov\left( \overline{X}_{\alpha }(t),\overline{X}_{\alpha }(s)\right) &=&%
\frac{\mathcal{\theta }}{\gamma }e^{-\gamma |\mathcal{T}_{\alpha }(t)-%
\mathcal{T}_{\alpha }(s)|}=\frac{\mathcal{\theta }}{\gamma }e^{\frac{1}{2}%
\left[ \log E_{\alpha ,1}(-\gamma 2^{\alpha }t^{\alpha })-\log E_{\alpha
,1}(-\gamma 2^{\alpha }s^{\alpha })\right] } \\
&=&\frac{\mathcal{\theta }}{\gamma }\sqrt{\frac{E_{\alpha ,1}(-\gamma
2^{\alpha }t^{\alpha })}{E_{\alpha ,1}(-\gamma 2^{\alpha }s^{\alpha })}}.
\end{eqnarray*}
\end{remark}

\subsection{Fractional OU process}

In order to overcome the anomalous asymptotic behavior of the previous
model, we start by defining the fractional Ornstein-Uhlenbeck process in the
stationary case, by means of the following preliminary result.

\begin{lemma}
\label{L11}
The function
\begin{equation*}
f(s)=E_{\alpha ,1}(-\gamma |s|^{\alpha })\text{, }s\in \mathbb{R},
\end{equation*}%
is positive definite.
\end{lemma}

\begin{proof}
We apply the Schoenberg's characterization: let $\varphi :[0,+\infty
)\rightarrow \mathbb{R}$ be a continuous function, then $\Phi (\cdot
)=\varphi (\left\Vert \cdot \right\Vert ^{2})$ is positive definite and
radial on $\mathbb{R}^{d},$ for any $d$, if and only if $\varphi (\cdot )$
is CM on $[0,+\infty )$. Then, we note that the complete monotonicity of the
Mittag-Leffler function of negative argument and that the composition of a
CM functional and a Bernstein function is again CM (see \cite{MAI} for
details). Since $s^{\alpha /2}$ is a Bernstein function, the result follows
by considering that $E_{\alpha ,1}(-\gamma s^{\alpha /2})$ is CM.
\end{proof}

With this at hand we can give the following definition:

\begin{definition}
\label{Def.12}
\textbf{(Fractional stationary OU) }Let $\overline{\mathcal{Y}}_{\alpha
}:=\left\{ \overline{\mathcal{Y}}_{\alpha }(t),t\geq 0\right\} $ be defined
as a Gaussian process with $\mathbb{E}\overline{\mathcal{Y}}_{\alpha }(t)=0$%
, for any $t\geq 0$ and
\begin{equation}
r(s):=Cov\left( \overline{\mathcal{Y}}_{\alpha }(t),\overline{\mathcal{Y}}%
_{\alpha }(t+s)\right) =\frac{\mathcal{\theta }}{\gamma }E_{\alpha
,1}(-\gamma |s|^{\alpha }),  \label{cov2}
\end{equation}%
for any $s\in \mathbb{R},$ $\gamma >0,$ $\alpha \in (0,1].$
\end{definition}

\begin{remark}
Long-range dependent processes with autocorrelation function decaying as a
Mittag-Leffler have been already studied in \cite{BAR}, where they are
obtained, in a completely different way, i.e. by superposition of OU
processes.
\end{remark}

Then, for the spectral density of the fractional stationary OU defined
above, we prove the following result.

\begin{lemma}
The spectral density of the process $\overline{\mathcal{Y}}_{\alpha }$ is
given by
\begin{equation}
\mathcal{S}_{\overline{\mathcal{Y}}_{\alpha }}(\omega )=\frac{\mathcal{%
\theta }}{\sqrt{\pi }\gamma }H_{2,3}^{2,1}\left[ \left. \frac{|\omega
|^{\alpha }}{2^{\alpha }\gamma }\right\vert
\begin{array}{ccc}
(0,1) & \left( 1-\alpha ,\alpha \right) & \quad \\
(\frac{1}{2}-\frac{\alpha }{2},\frac{\alpha }{2}) & (0,1) & (1-\frac{\alpha
}{2},\frac{\alpha }{2})%
\end{array}%
\right] ,\qquad \omega \in \mathbb{R}\backslash \{0\},  \label{cov4}
\end{equation}%
where $H_{p,q}^{m,n}$ is the H-function defined in (\ref{hh}). Then its
spectral representation reads%
\begin{equation}
\overline{\mathcal{Y}}_{\alpha }(t)=\int_{-\infty }^{+\infty }\sqrt{\mathcal{%
S}_{\overline{\mathcal{Y}}_{\alpha }}(\omega )}e^{it\omega }dW(\omega ).
\label{rep}
\end{equation}
\end{lemma}

\begin{proof}
From (\ref{cov2}) we can write, for any $t\geq 0$ and $\omega \in \mathbb{R}$%
,%
\begin{eqnarray}
\mathcal{S}_{\overline{\mathcal{Y}}_{\alpha }}(\omega ) &=&\frac{1}{2\pi }%
\int_{-\infty }^{+\infty }e^{-i\omega s}Cov(\overline{\mathcal{Y}}_{\alpha
}(t),\overline{\mathcal{Y}}_{\alpha }(t+s))ds  \label{cov3} \\
&=&\frac{\mathcal{\theta }}{\pi \gamma }\int_{0}^{+\infty }\cos (|\omega
|s)E_{\alpha ,1}(-\gamma s^{\alpha })ds  \notag \\
&=&[\text{by (1.12.65) in \cite{KIL}}]  \notag \\
&=&\frac{\mathcal{\theta }}{\pi \gamma }\int_{0}^{+\infty }\cos (|\omega
|s)H_{1,2}^{1,1}\left[ \left. \gamma s^{\alpha }\right\vert
\begin{array}{c}
\left( 0,1\right) \\
\left( 0,1\right) (0,\alpha )%
\end{array}%
\right] ds.  \notag
\end{eqnarray}%
In order to evaluate the previous cosine transform, we apply formula (2.50),
for $\rho =1,$ of \cite{MAT}, i.e.%
\begin{equation*}
\int_{0}^{+\infty }\cos (ax)H_{p,q}^{m,n}\left[ \left. bx^{\sigma
}\right\vert
\begin{array}{c}
(a_{p},A_{p}) \\
(b_{q},B_{q})%
\end{array}%
\right] dt=\frac{\sqrt{\pi }}{a}H_{p+2,q}^{m,n+1}\left[ \left. b\left( \frac{%
2}{a}\right) ^{\sigma }\right\vert
\begin{array}{ccc}
(\frac{1}{2},\frac{\sigma }{2}) & (a_{p},A_{p}) & \left( 0,\frac{\sigma }{2}%
\right) \\
(b_{q},B_{q}) & \quad & \quad%
\end{array}%
\right] ,
\end{equation*}%
for $a,\alpha ,\sigma >0,$ $b\in \mathbb{C},$ $\mathcal{R}\left[ 1+\sigma
\max_{1\leq j\leq n}\left\{ \frac{a_{j}-1}{A_{j}}\right\} \right] <1$ and $%
\mathcal{R}\left[ 1+\sigma \min_{1\leq j\leq m}\left\{ \frac{b_{j}}{B_{j}}%
\right\} \right] >0$; $arg|a|<\pi \alpha /2$, where $\alpha
=\sum_{j=1}^{m}B_{j}-\sum_{j=m+1}^{q}B_{j}+\sum_{j=1}^{n}A_{j}-%
\sum_{j=n+1}^{p}A_{j}.$

It is easy to check that the corresponding conditions are all satisfied,
since $1+\alpha \left\{ \frac{a_{1}-1}{A_{1}}\right\} =1-\alpha <1$, $%
1+\alpha \frac{b_{1}}{B_{1}}=1>0$ and $\alpha =B_{1}-B_{2}+A_{1}=2-\alpha
>0. $ Then (\ref{cov3}) becomes%
\begin{eqnarray}
\mathcal{S}_{\overline{\mathcal{Y}}_{\alpha }}(\omega ) &=&\frac{\mathcal{%
\theta }}{\sqrt{\pi }\gamma |\omega |}H_{3,2}^{1,2}\left[ \left. \frac{%
2^{\alpha }\gamma }{|\omega |^{\alpha }}\right\vert
\begin{array}{ccc}
(\frac{1}{2},\frac{\alpha }{2}) & (0,1) & (0,\frac{\alpha }{2}) \\
(0,1) & (0,\alpha ) & \quad%
\end{array}%
\right]  \label{tr} \\
&=&[\text{by (1.12.47) in \cite{KIL}}]  \notag \\
&=&\frac{\mathcal{\theta }}{\sqrt{\pi }\gamma |\omega |}H_{2,3}^{2,1}\left[
\left. \frac{|\omega |^{\alpha }}{2^{\alpha }\gamma }\right\vert
\begin{array}{ccc}
(1,1) & (1,\alpha ) & \quad \\
(\frac{1}{2},\frac{\alpha }{2}) & (1,1) & (1,\frac{\alpha }{2})%
\end{array}%
\right] ,  \notag
\end{eqnarray}%
which coincides with (\ref{cov4}), by applying formula (1.12..45) in \cite%
{KIL}, for $\sigma =-1/\alpha $. The existence condition in \cite{MAT} is
satisfied, for any $\omega \neq 0,$ by (\ref{cov4}), since $\alpha =\alpha
/2>0,$ by choosing the contour $\mathcal{L}=\mathcal{L}_{i\gamma \infty },$
since in this case $\mu =\sum_{j=1}^{q}B_{j}-\sum_{j=1}^{p}A_{j}=0$ (see
Theorem 1.1, case 7, in \cite{MAT}). Finally, the function (\ref{cov4}) is
positive, since (\ref{cov2}) is positive definite, by the Bochner's theorem.
The representation (\ref{rep}) easily follows.
\end{proof}

\begin{remark}
We can check that in the special case $\alpha =1$ formula (\ref{cov4})
reduces to the well-known spectral density of the OU process:%
\begin{eqnarray*}
\mathcal{S}_{\overline{\mathcal{Y}}_{\alpha }}(\omega ) &=&\frac{\mathcal{%
\theta }}{2\gamma \sqrt{\pi }}H_{2,3}^{2,1}\left[ \left. \frac{|\omega |}{%
2\gamma }\right\vert
\begin{array}{ccc}
(0,1) & \left( 0,1\right) & \quad \\
(0,\frac{1}{2}) & (0,1) & (\frac{1}{2},\frac{1}{2})%
\end{array}%
\right] \\
&=&\frac{\mathcal{\theta }}{2\gamma \sqrt{\pi }}\frac{1}{2\pi i}\int_{%
\mathcal{L}}\left( \frac{|\omega |}{2\gamma }\right) ^{-s}\frac{\Gamma
\left( \frac{s}{2}\right) \Gamma (1-s)}{\Gamma \left( \frac{1}{2}-\frac{s}{2}%
\right) }ds \\
&=&[\text{by the duplication formula}] \\
&=&\frac{\mathcal{\theta }}{2\pi \gamma }\frac{1}{2\pi i}\int_{\mathcal{L}%
}\left( \frac{|\omega |}{\gamma }\right) ^{-s}\Gamma \left( \frac{s}{2}%
\right) \Gamma \left( 1-\frac{s}{2}\right) ds \\
&=&\frac{\mathcal{\theta }}{2\pi \gamma }H_{1,1}^{1,1}\left[ \left. \frac{%
|\omega |}{\gamma }\right\vert
\begin{array}{c}
\left( 0,\frac{1}{2}\right) \\
\left( 0,\frac{1}{2}\right)%
\end{array}%
\right] \\
&=&[\text{see (4.8) in \cite{VEL}}] \\
&=&\frac{\mathcal{\theta }\gamma }{\pi }\frac{1}{\gamma ^{2}+\omega ^{2}}.
\end{eqnarray*}%
for $\omega \in \mathbb{R}.$
\end{remark}

In order to derive the dependence properties of the process $\overline{%
\mathcal{Y}}_{\alpha }$ we apply Theorem 1.2, p.19, in \cite{MAT}: the
asymptotic behavior, for $\left\vert \omega \right\vert \rightarrow 0$, of
the spectral density (\ref{tr}) is then given by%
\begin{equation}
\mathcal{S}_{\overline{\mathcal{Y}}_{\alpha }}(\omega )=\frac{1}{|\omega |}%
O(\left\vert \omega \right\vert ^{\alpha })=O(\left\vert \omega \right\vert
^{\alpha -1}),\qquad \left\vert \omega \right\vert \rightarrow 0,
\label{big}
\end{equation}%
since $c=\min \left[ 1/\alpha ,1\right] =1.$ As a consequence, we can
conclude that the process $\overline{\mathcal{Y}}_{\alpha }$ exhibits
long-range dependence under the assumption that $\alpha \in (0,1)$, while
for $\alpha =1$ we have $\frac{\mathcal{S}_{\mathcal{Y}}(\omega )}{%
\left\vert \omega \right\vert ^{a }}\rightarrow \infty $, for any $a \in
(0,1),$ as it is well-known in the standard OU case. The same conclusion
could be drawn form the covariance function (\ref{cov2}), taking into
account (\ref{asy}), so that $r(s)\sim C_{\alpha }s^{-\alpha },$ for some
constant $C_{\alpha }$ and $s\rightarrow +\infty .$ Thus the process is
long-memory of order $1- \alpha$, both in the covariance and in the spectral
density sense.

We prove the following properties for the trajectories of the fractional
stationary OU process.

\begin{theorem}
There exists a version of $\overline{\mathcal{Y}}_{\alpha }$ with continuous
trajectories and the latter are not differentiable in the $L^{2}$-norm.

\begin{proof}
In order to apply the Kolmogorov continuity theorem, we only need to check
that there exist $p,\eta >1$, such that, for a constant $c$ and for any $%
t_{1},t_{2}\in I$, the following inequality holds:%
\begin{equation*}
\mathbb{E}\left\vert \overline{\mathcal{Y}}_{\alpha }(t_{2})-\overline{%
\mathcal{Y}}_{\alpha }(t_{1})\right\vert ^{p}\leq c\left\vert
t_{2}-t_{1}\right\vert ^{\eta }.
\end{equation*}%
Let $\left\lceil x\right\rceil $ denote the upper integer part of $x\in
\mathbb{R}$, then, by choosing $p=\left\lceil 2/\alpha \right\rceil $, $\eta
=\frac{\alpha p}{2}>1$, the previous inequality is verified; indeed
\begin{eqnarray*}
\mathbb{E}\left\vert \overline{\mathcal{Y}}_{\alpha }(t_{2})-\overline{%
\mathcal{Y}}_{\alpha }(t_{1})\right\vert ^{p} &=&\frac{2^{p}\Gamma ((p+1)/2)%
}{\sqrt{\pi }\gamma ^{p/2}}\left[ \mathbb{E}\left[ \overline{\mathcal{Y}}%
_{\alpha }(t_{2})-\overline{\mathcal{Y}}_{\alpha }(t_{1})\right] ^{2}\right]
^{p/2} \\
&=&C_{p}\left[ 1-E_{\alpha ,1}(-\gamma \left\vert t_{2}-t_{1}\right\vert
^{\alpha })\right] ^{p/2} \\
&=&\gamma ^{p/2}C_{p}\left\vert t_{2}-t_{1}\right\vert ^{\alpha p/2}\left[
E_{\alpha ,\alpha +1}(-\gamma \left\vert t_{2}-t_{1}\right\vert ^{\alpha })%
\right] ^{p/2} \\
&\leq &\frac{\gamma ^{p/2}C_{p}}{\Gamma (\alpha +1)^{p/2}}\left\vert
t_{2}-t_{1}\right\vert ^{\alpha p/2},
\end{eqnarray*}%
(where $C_{p}=2^{p}\theta ^{p/2}\Gamma ((p+1)/2)/\gamma ^{p/2}\sqrt{\pi })$,
since $E_{\alpha ,\alpha +1}(-\gamma \left\vert t_{2}-t_{1}\right\vert
^{\alpha })\leq 1/\Gamma (\alpha +1)$. Indeed $E_{\alpha ,\alpha
+1}(0)=1/\Gamma (\alpha +1)$ and the function is completely monotone and
non-increasing (see \cite{GOR}, Sec. 4.10.2)$.$ Moreover we have that
\begin{equation*}
\mathbb{E}\left( \frac{\overline{\mathcal{Y}}_{\alpha }(t_{2})-\overline{%
\mathcal{Y}}_{\alpha }(t_{1})}{t_{2}-t_{1}}\right) ^{2}=2\mathcal{\theta }%
\left\vert t_{2}-t_{1}\right\vert ^{\alpha -2}E_{\alpha ,\alpha +1}(-\gamma
\left\vert t_{2}-t_{1}\right\vert ^{\alpha })\rightarrow +\infty ,
\end{equation*}%
for $t_{2}\rightarrow t_{1}$ and $\alpha \in (0,1],$ by considering that $%
lim_{x\rightarrow 0}E_{\alpha ,\alpha +1}(-x^{\alpha })=1/\Gamma (\alpha
+1). $ Alternatively, we could notice that the covariance function can be
expanded as follows: $r(\tau )=r(0)-\frac{\gamma |\tau |^{\alpha }}{\Gamma
(\alpha +1)}+o(|\tau |^{\alpha }).$
\end{proof}
\end{theorem}

We can now define a non-stationary version of the fractional OU process,
whose density function coincides with the solution of the fractional
Fokker-Planck equation (\ref{ch}), in Fourier space. We start by the
following preliminary result.

\begin{lemma}
The function
\begin{equation*}
f(s)=\sqrt{E_{\alpha ,1}(-ks^{\alpha })}\text{, }s\in \mathbb{R},\text{ }k>0
\end{equation*}%
is positive definite.
\end{lemma}

\begin{proof}
Similarly to Lemma \ref{L2}, the result follows by proving that $\sqrt{E_{\alpha
,1}(-\gamma s^{\alpha /2})}$ is CM. But, as can be checked directly, by
differentiating, if $f(t)$ is a CM function, then the same is true for $%
\sqrt{f(t)}.$
\end{proof}

\begin{definition}
\textbf{(Fractional OU) }Let $\left( \Omega ,\mathcal{F},\mathcal{F}%
_{t},P\right) $ be a probability space with filtration $\mathcal{F}_{t}$ and
let$\ \mathcal{Y}_{\alpha }:=\left\{ \mathcal{Y}_{\alpha }(t),t\geq
0\right\} $ be defined as follows
\begin{equation*}
\mathcal{Y}_{\alpha }(t)=\overline{\mathcal{Y}}_{\alpha }(t)-\sqrt{E_{\alpha
,1}(-\gamma (2t)^{\alpha })}\mathcal{Z},
\end{equation*}%
where $\left( \overline{\mathcal{Y}}_{\alpha }(t),\mathcal{Z}\right) $ is a
Gaussian centered random vector with covariance matrix
\begin{equation*}
\sum =\left(
\begin{array}{cc}
\frac{\mathcal{\theta }}{\gamma } & \frac{\mathcal{\theta }}{\gamma }\sqrt{%
E_{\alpha ,1}(-\gamma (2t)^{\alpha })} \\
\frac{\mathcal{\theta }}{\gamma }\sqrt{E_{\alpha ,1}(-\gamma (2t)^{\alpha })}
& \frac{\mathcal{\theta }}{\gamma }%
\end{array}%
\right) ,
\end{equation*}%
for any $t\geq 0.$
\end{definition}

\begin{remark}
The fractional OU process $\mathcal{Y}_{\alpha }$ is centered, Gaussian with
\begin{equation}
Cov\left( \mathcal{Y}_{\alpha }(t),\mathcal{Y}_{\alpha }(s)\right) =\frac{%
\mathcal{\theta }}{\gamma }\left[ E_{\alpha ,1}(-\gamma |t-s|^{\alpha })-%
\sqrt{E_{\alpha ,1}(-\gamma (2t)^{\alpha })E_{\alpha ,1}(-\gamma
(2s)^{\alpha })}\right] ,  \label{cov1}
\end{equation}%
for $s,t\geq 0,$ $\alpha \in (0,1],$ and $\mathcal{\theta },\gamma >0.$
Since $\mathcal{Y}_{\alpha }$ is a zero-mean Gaussian process, its
finite-dimensional distributions are completely characterized by its
autocovariance function. The variance of the process reads%
\begin{equation*}
Var\mathcal{Y}_{\alpha }(t)=\frac{\mathcal{\theta }}{\gamma }\left[
1-E_{\alpha ,1}(-\gamma (2t)^{\alpha })\right]
\end{equation*}%
which is positive since $E_{\alpha ,1}(-x^{\alpha })\leq 1,$ for $x\geq 0.$
Thus, the one-dimensional density of the process $\mathcal{Y}_{\alpha }$
coincides with the solution to the fractional FP equation (\ref{ch}) and its
characteristic function is given in (\ref{ch2}). By considering (\ref{asy}),
we notice that the process $\mathcal{Y}_{\alpha }$ is a mean reverting
process, for any $\alpha $, since $\lim_{t\rightarrow \infty }Var\mathcal{Y}%
_{\alpha }(t)=\mathcal{\theta }/\gamma $. Moreover from (\ref{cov1}) it is
evident that the process is non-Markovian for $\alpha \neq 1$, since
\begin{equation*}
\rho _{\mathcal{Y}_{\alpha }}(s,t)=\frac{E_{\alpha ,1}(-\gamma |t-s|^{\alpha
})-\sqrt{E_{\alpha ,1}(-\gamma (2t)^{\alpha })E_{\alpha ,1}(-\gamma
(2s)^{\alpha })}}{\sqrt{1-E_{\alpha ,1}(-\gamma (2t)^{\alpha })}\sqrt{%
1-E_{\alpha ,1}(-\gamma (2s)^{\alpha })}}\neq \rho _{\mathcal{Y}_{\alpha
}}(s,h)\rho _{\mathcal{Y}_{\alpha }}(h,t),
\end{equation*}%
for $s,h,t\geq 0.$ Only for $\alpha =1,$ the sufficient condition for the
Markov property (in the Gaussian case) is verified: indeed the process $%
\mathcal{Y}_{1}$ coincides with the standard OU process, since
\begin{equation*}
Cov\left( \mathcal{Y}_{1}(t),\mathcal{Y}_{1}(s)\right) =\frac{\mathcal{%
\theta }}{\gamma }\left[ e^{-\gamma |t-s|}-e^{-\gamma |t+s|}\right]
=Cov\left( X(t),X(s)\right) .
\end{equation*}%
Thus, for $\alpha \neq 1,$ the process must also have dependent increments,
otherwise the Markovian property would follow from Gaussianity.
\end{remark}

\begin{theorem}
The cumulant generating function $C_{\mathcal{Y}_{\alpha }}(\eta ,t):=\log
\mathbb{E}e^{\eta \mathcal{Y}_{\alpha }(t)}$ of $\mathcal{Y}_{\alpha }$,
i.e.
\begin{equation}
C_{\mathcal{Y}_{\alpha }}(\eta ,t)=\frac{\eta ^{2}\mathcal{\theta }}{2\gamma
}\left[ 1-E_{\alpha ,1}(-\gamma (2t)^{\alpha })\right] .  \label{cc2}
\end{equation}%
satisfies the following fractional equation%
\begin{equation}
D_{t}^{\alpha }u(\eta ,t)=-2^{\alpha }\gamma u(\eta ,t)+2^{\alpha -1}\eta
^{2}\mathcal{\theta },  \label{cc}
\end{equation}%
with $u(\eta ,0)=0.$
\end{theorem}

\begin{proof}
Let $\widetilde{f}(s):=\mathcal{L}\left\{ f(t);s\right\} =\int_{0}^{+\infty
}e^{-st}f(t)dt$ denote the Laplace transform of $f(\cdot ).$ By considering
the well-known formula (see \cite{KIL}, Lemma 2.24)%
\begin{equation*}
\mathcal{L}\left\{ D_{t}^{\alpha }f(t);s\right\} =s^{\alpha }\widetilde{f}%
(s)-\sum_{j=0}^{m-1}s^{\alpha -j-1}\left. \frac{d^{n}}{dt^{n}}%
f(t)\right\vert _{t=0},\qquad m-1<\alpha \leq m,
\end{equation*}%
the Laplace transform of (\ref{cc}) reads%
\begin{equation*}
s^{\alpha }\widetilde{u}(\eta ,s)-s^{\alpha -1}u(\eta ,0)=-2^{\alpha }\gamma
\widetilde{u}(\eta ,s)+\frac{2^{\alpha -1}\eta ^{2}\mathcal{\theta }}{s}.
\end{equation*}%
Then, by the initial condition, we get%
\begin{equation*}
\widetilde{u}(\eta ,s)=\frac{2^{\alpha -1}\eta ^{2}\mathcal{\theta }s^{-1}}{%
s^{\alpha }+2^{\alpha }\gamma },
\end{equation*}%
which, by means of (\ref{lap}), gives%
\begin{equation*}
u(\eta ,t)=2^{\alpha -1}\eta ^{2}\mathcal{\theta }t^{\alpha }E_{\alpha
,\alpha +1}(-\gamma (2t)^{\alpha })=2^{\alpha -1}\eta ^{2}\mathcal{\theta }%
t^{\alpha }\sum_{j=0}^{\infty }\frac{(-2^{\alpha }\gamma t^{\alpha })^{j}}{%
\Gamma (\alpha j+\alpha +1)}
\end{equation*}%
which, after some algebra, coincides with (\ref{cc2}). The latter satisfies
the initial conditions since $E_{\alpha ,1}(0)=1$.
\end{proof}

\begin{remark}
We notice that a similar result holds for the moment generating function of
the time-changed Brownian motion (time-fractional diffusion), i.e. $W(%
\mathcal{L}_{\alpha }(t))$, where $W$ denotes a standard Brownian motion and
$\mathcal{L}_{\alpha }$ the inverse of an independent $\alpha $-stable
subordinator (see \cite{ORS}). Indeed for $\mathcal{M}(\theta ,t):=\mathbb{E}%
e^{\theta W(\mathcal{L}_{\alpha }(t))}$ we have that $\mathcal{M}(\theta
,t)=E_{\alpha ,1}(\frac{\theta ^{2}}{2}t^{\alpha }).$
\end{remark}

\section{Further generalizations}

We generalize the results of the previous sections by considering the
Fokker-Plank equation (in Fourier space) with a convolution-type operator
defined by means of a general Bernstein function (not necessarily
fractional).

\begin{definition}
\label{Def.22}
Let $u:\mathbb{R}^{+}\rightarrow \mathbb{R}$ be an absolutely continuous
function and $D_{x}^{g}$ be the convolution-type derivative defined in (4),
then
\begin{equation}
\mathcal{L}_{x}^{g}u(x):=u(x)D_{x}^{g}\log u(x),\qquad x\in \mathbb{R}^{+}%
\text{.}  \label{conv1}
\end{equation}
\end{definition}

We start by proving the following preliminary result which concerns the
transition density of the inverse subordinator: let us denote by $\mathcal{A}%
^{g}(t),$ $t\geq 0,$ the subordinator with Laplace exponent $g(s)$, i.e.
such that%
\begin{equation*}
\mathbb{E}e^{-s\mathcal{A}^{g}(t)}=e^{-g(s)t}.
\end{equation*}%
Let moreover $L^{g}(t),$ $t\geq 0,$ be its inverse, i.e.
\begin{equation*}
L^{g}(t)=\inf \left\{ s>0:\mathcal{A}^{g}(s)>t\right\} ,\qquad s,t>0
\end{equation*}%
and $l_{g}(x,t)=\Pr \left\{ L^{g}(t)\in dx\right\} $ be its transition
density.

\begin{lemma}
\label{L23}
The following initial-value problem%
\begin{equation}
D_{t}^{g}u(t)=-ku(t),\qquad t,k\geq 0,  \label{rel}
\end{equation}%
with $u(0)=1$, is satisfied by the Laplace transform of the inverse
subordinator density $l_{g}(x,t)$, i.e. by
\begin{equation}
\widetilde{l}_{g}(k,t)=\int_{0}^{+\infty }e^{-kx}l_{g}(x,t)dx=\mathbb{E}%
e^{-kL^{g}(t)}.  \label{rel2}
\end{equation}
\end{lemma}

\begin{proof}
By considering the Laplace transform of the inverse subordinator given in (%
\cite{TOA}), Proposition 3.2, we get%
\begin{equation*}
\widetilde{u}(s)=\frac{g(s)}{s}\int_{0}^{+\infty }e^{-kx}e^{-g(s)x}dx=\frac{%
g(s)}{s}\frac{1}{k+g(s)}.
\end{equation*}%
On the other hand, from (\ref{rel}), we get, by (\ref{lapconv}) together
with the initial condition,%
\begin{equation*}
g(s)\int_{0}^{+\infty }e^{-st}\widetilde{l}_{g}(k,t)dt-\frac{g(s)}{s}%
=-k\int_{0}^{+\infty }e^{-st}\widetilde{l}_{g}(k,t)dt.
\end{equation*}
\end{proof}

By means of the previous result, we can prove the following generalization
of Theorem \ref{Thm6}.

\begin{theorem}
Let $\mathcal{L}_{t}^{g}$ be the operator defined in Def.\ref{Def.22}, then the
solution of the generalized \textit{Fokker-Planck equation (in Fourier space)%
}%
\begin{equation}
\mathcal{L}_{t}^{g}\widehat{u}(\xi ,t)=-\frac{\gamma }{2^{1-\alpha }}\xi
\frac{\partial }{\partial \xi }\widehat{u}(\xi ,t)-\frac{\theta }{%
2^{1-\alpha }}\xi ^{2}\widehat{u}(\xi ,t),\qquad \xi \in \mathbb{R},\text{ }%
t\geq 0,\text{ }\alpha \in (0,1],\text{ }D,\gamma >0,  \label{rel3}
\end{equation}%
with initial condition $\widehat{u}(\xi ,0)=1$, for any $t\geq 0,$ is given
by%
\begin{equation}
\widehat{u}(\xi ,t)=\exp \left\{ -\frac{\mathcal{\theta }\xi ^{2}}{2\gamma }%
\left[ 1-\widetilde{l}_{g}(\gamma ,t)\right] \right\} .  \label{rel4}
\end{equation}
\end{theorem}

\begin{proof}
By applying Lemma \ref{L23}, we get
\begin{eqnarray*}
D_{t}^{g}\log \widehat{u}(\xi ,t) &=&-\frac{\mathcal{\theta }\xi ^{2}}{%
2\gamma }D_{t}^{g}\left[ 1-\widetilde{l}_{g}(\gamma ,t)\right] \\
&=&-\mathcal{\theta }\xi \widetilde{l}_{g}(\gamma ,t) \\
&=&-2\gamma \log \widehat{u}(\xi ,t)-\mathcal{\theta }\xi ^{2}
\end{eqnarray*}%
and thus%
\begin{eqnarray*}
\widehat{u}(\xi ,t)D_{t}^{g}\log \widehat{u}(\xi ,t) &=&\mathcal{\theta }\xi
^{2}\left[ 1-\widetilde{l}_{g}(\gamma ,t)\right] \widehat{u}(\xi ,t)-%
\mathcal{\theta }\xi ^{2}\widehat{u}(\xi ,t) \\
&=&-\xi \gamma \frac{\partial }{\partial \xi }\widehat{u}(\xi ,t)-\mathcal{%
\theta }\xi ^{2}\widehat{u}(\xi ,t),
\end{eqnarray*}%
which coincides with (\ref{rel3}). The initial condition is verified since $%
\widetilde{l}_{g}(\gamma ,0)=\int_{0}^{+\infty }e^{-\gamma
x}l_{g}(x,0)dx=\int_{0}^{+\infty }e^{-\gamma x}\delta (x)dx=1.$
\end{proof}

In order to define the generalized time-changed OU process, we check that $%
\mathcal{T}_{g}(\cdot ):=-\frac{1}{2\gamma }\log \widetilde{l}_{g}(\gamma
,\cdot )$ is an increasing and continuous function. Let us denote by $%
h_{g}(x,t)$ the density of the subordinator $\mathcal{A}^{g}$, then we get,
for any $t>0,$
\begin{eqnarray}
\widetilde{l}_{g}(\gamma ,t) &=&\int_{0}^{+\infty }e^{-\gamma x}l_{g}(x,t)dx
\label{rrr} \\
&=&\int_{0}^{+\infty }e^{-\gamma x}\frac{\partial }{\partial x}P\left\{
L^{g}(t)<x\right\} dx  \notag \\
&=&\int_{0}^{+\infty }e^{-\gamma x}\frac{\partial }{\partial x}P\left\{
\mathcal{A}^{g}(x)>t\right\} dx  \notag \\
&=&\left[ e^{-\gamma x}P\left\{ \mathcal{A}^{g}(x)>t\right\} \right]
_{x=0}^{+\infty }+\gamma \int_{0}^{+\infty }e^{-\gamma x}P\left\{ \mathcal{A}%
^{g}(x)>t\right\} dx  \notag \\
&=&\gamma \int_{0}^{+\infty }e^{-\gamma x}\left( \int_{t}^{+\infty
}h_{g}(s,x)ds\right) dx  \notag \\
&=&\gamma \int_{t}^{+\infty }\left( \int_{0}^{+\infty }e^{-\gamma
x}h_{g}(s,x)dx\right) ds,  \notag
\end{eqnarray}%
which is clearly decreasing, since the quantity inside brackets is positive.
Continuity follows from Proposition 3.2 in \cite{TOA}.

\begin{definition}
\label{Def.25}
Let $\left( \Omega ,\mathcal{F},\mathcal{F}_{t},P\right) $ be a probability
space with filtration $\mathcal{F}_{t}$ and let$\ \mathcal{X}_{g}:=\left\{
\mathcal{X}_{g}(t),t\geq 0\right\} $ be defined as $\mathcal{X}_{g}(t):=X(%
\mathcal{T}_{g}(t)),$ where $X$ is the standard OU process and $\mathcal{T}%
_{g}(t)=-\frac{1}{2\gamma }\log \widetilde{l}_{g}(\gamma ,t).$
\end{definition}

\begin{remark}
It is easy to check that $\mathcal{X}_{g}$ is a Gaussian process, with $%
\mathbb{E}\mathcal{X}_{g}(t)=0$ and%
\begin{equation*}
Cov\left( \mathcal{X}_{g}(t),\mathcal{X}_{g}(s)\right) =\frac{\mathcal{%
\theta }}{\gamma }\sqrt{\frac{\widetilde{l}_{g}(\gamma ,t\vee s)}{\widetilde{%
l}_{g}(\gamma ,t\wedge s)}}\left[ 1-\widetilde{l}_{g}(\gamma ,t\wedge s)%
\right] ,
\end{equation*}%
for $\ t,s\geq 0$, $\gamma >0$ and for any Bernstein function $g.$ Moreover,
the process $\mathcal{X}_{g}$ is related to the Brownian motion by the
following equality in distribution
\begin{equation}
\mathcal{X}_{g}(t)\overset{d}{=}\sqrt{\frac{\theta }{\gamma }\widetilde{l}%
_{g}(\gamma ,t)}W\left( \frac{1}{\widetilde{l}_{g}(\gamma ,t)}-1\right) .
\label{xg}
\end{equation}
\end{remark}

The definition of fractional OU processes given in section 3.2 can be
analogously generalized, under the assumption that $\widetilde{l}_{g}(\gamma
,|\cdot |)$ is a positive definite function, both in the stationary and
non-stationary cases. From (\ref{rrr}), we can write the following
relationship between the derivatives of $\widetilde{l}_{g}(\gamma ,t)$ and $%
\widetilde{h}_{g}(t,\gamma )$:%
\begin{equation}
\frac{\partial ^{r}}{\partial t^{r}}\widetilde{l}_{g}(\gamma ,t)=-\gamma
\frac{\partial ^{r}}{\partial t^{r}}\widetilde{h}_{g}(t,\gamma ),
\label{rrr2}
\end{equation}%
for any $r=1,2,...$. Therefore, by steps similar to Lemma \ref{L11}, we can
conclude that $\widetilde{l}_{g}(\gamma ,|\cdot |)$ is a positive definite
function if $\widetilde{l}_{g}(\gamma ,\cdot )$ is CM and (from (\ref{rrr2}%
)) this holds if the subordinator $\mathcal{A}^{g}$ is such that the
time-Laplace transform of its density, i.e. $\widetilde{h}_{g}(\cdot ,\gamma
),$ is a Bernstein function.

\begin{definition}
\label{Def.27}
Let $\mathcal{A}^{g}$ be a subordinator such that the time-Laplace transform
of its density, i.e. $\widetilde{h}_{g}(\cdot ,\gamma ),$ is a Bernstein
function. Let $\overline{\mathcal{Y}}_{g}:=\left\{ \overline{\mathcal{Y}}%
_{g}(t),t\geq 0\right\} $ be a Gaussian process with $\mathbb{E}\overline{%
\mathcal{Y}}_{g}(t)=0$, for any $t\geq 0$ and
\begin{equation}
r(s):=Cov\left( \overline{\mathcal{Y}}_{g}(t),\overline{\mathcal{Y}}%
_{g}(t+s)\right) =\frac{\mathcal{\theta }}{\gamma }\widetilde{l}_{g}(\gamma
,|s|),
\end{equation}%
for any $s\in \mathbb{R},$ $\gamma >0.$
\end{definition}

We can also give a condition for long-range dependence of $\overline{%
\mathcal{Y}}_{g}$ in terms of the tail's probability of the subordinator,
i.e. $P\left\{ \mathcal{A}^{g}(x)>t\right\} ,$ for $t\rightarrow +\infty .$
Indeed it is evident from (\ref{rrr}) that if
\begin{equation}
P\left\{ \mathcal{A}^{g}(x)>t\right\} \sim K_{x}t^{1-H},  \label{lrd}
\end{equation}%
for some $H\in (0,1)$ and a constant $K_{x}>0,$ then $r(s)=\frac{\mathcal{%
\theta }}{\gamma }\widetilde{l}_{g}(\gamma ,s)\sim K_{\gamma }^{\prime
}s^{1-H}$ for $s\rightarrow +\infty $ and for $K_{\gamma }^{\prime }>0.$
Thus, under condition (\ref{lrd}), the process is long-memory of order $H$,
in the covariance sense.

\begin{definition}
Let $\overline{\mathcal{Y}}_{g}$ be defined as in Def.\ref{Def.27} and let
\begin{equation*}
\mathcal{Y}_{g}(t)=\overline{\mathcal{Y}}_{g}(t)-\sqrt{\widetilde{l}%
_{g}(\gamma ,2t)}\mathcal{Z},
\end{equation*}%
where $\left( \overline{\mathcal{Y}}_{g}(t),\mathcal{Z}\right) $ is a
Gaussian centered random vector with covariance matrix
\begin{equation*}
\sum =\left(
\begin{array}{cc}
\frac{\mathcal{\theta }}{\gamma } & \frac{\mathcal{\theta }}{\gamma }\sqrt{%
\widetilde{l}_{g}(\gamma ,2t))} \\
\frac{\mathcal{\theta }}{\gamma }\sqrt{\widetilde{l}_{g}(\gamma ,2t)} &
\frac{\mathcal{\theta }}{\gamma }%
\end{array}%
\right) ,
\end{equation*}%
for any $t\geq 0.$
\end{definition}

\begin{remark}
The generalized OU process $\mathcal{Y}_{g}$ is centered, Gaussian, with
\begin{equation}
Cov\left( \mathcal{Y}_{g}(t),\mathcal{Y}_{g}(s)\right) =\frac{\mathcal{%
\theta }}{\gamma }\left[ \widetilde{l}_{g}(\gamma ,|t-s|)-\sqrt{\widetilde{l}%
_{g}(\gamma ,2t)\widetilde{l}_{g}(\gamma ,2s)}\right] ,
\end{equation}%
for $s,t\geq 0\ $and $\mathcal{\theta },\gamma >0.$ Since $\mathcal{Y}_{g}$
is a zero-mean Gaussian process, its finite-dimensional distributions are
completely characterized by its autocovariance function. The characteristic
function of $\mathcal{Y}_{g}$ coincides with (\ref{rel4}) and therefore the
one-dimensional density of the process $\mathcal{Y}_{g}$ coincides with the
solution to the generalized FP equation (\ref{rel3}). Moreover, $\mathcal{Y}%
_{g}$ is non-Markovian and mean-reverting, since $\lim_{t\rightarrow \infty
}Var\mathcal{Y}_{\alpha }(t)=\lim_{t\rightarrow \infty }\mathcal{\theta }%
/\gamma \left[ 1-\widetilde{l}_{g}(\gamma ,2t)\right] =\mathcal{\theta }%
/\gamma $, as can be checked in (\ref{rrr}).
\end{remark}

We now show how the results of sections 3.1 and 3.2 can be derived as
special cases and we also present another interesting special case. Other
examples could be derived analogously, by specifying the Bernstein function
and the corresponding L\'{e}vy measure.

\begin{example}
\textbf{(Stable case) }In the special case where $g(s)=s^{\alpha }$, for $%
\alpha \in (0,1],$ the process $\mathcal{A}^{g}$ coincides with the $\alpha $%
-stable subordinator with $\mathbb{E}e^{-s\mathcal{A}^{g}(t)}=e^{-s^{\alpha
}t}$ (whose L\'{e}vy density is $\overline{\nu }(x)=x^{-\alpha -1}/\Gamma
(m-\alpha )$). Formula (\ref{conv}) reduces to the Caputo derivative (\ref%
{ca}), for $m=1.$ Thus the operator defined in (\ref{conv1}) coincides with $%
\mathcal{L}_{t}^{\alpha }$ given in Def.\ref{Def.1}. It is immediate to check that, in
this special case, formula (\ref{rel2}) coincides with the Laplace transform
of the inverse $\alpha $-stable subordinator, i.e. $\widetilde{l}_{\alpha
}(\gamma ,t)=\int_{0}^{+\infty }e^{-\gamma x}l_{\alpha }(x,t)dx=E_{\alpha
,1}(-\gamma t^{\alpha })$ (see \cite{MEE2}). Therefore the result of Lemma
\ref{L23} reduces to the well-known property of the Mittag-Leffler function of
being the eigenfunction of the Caputo fractional derivative (see, e.g.,
Lemma 2.23 in \cite{KIL}). Correspondingly, Def.\ref{Def.27} reduces to Def.\ref{Def.12} and
formula (\ref{xg}) coincides with (\ref{xa}).
\end{example}

\begin{example}
\textbf{(Compound Poisson with exponential jumps) }Another interesting case
can be obtained by specializing the previous results to the case where $%
g(s)=s/(s+a),$ with $a>0.$ This Bernstein function corresponds to the
exponential L\'{e}vy density $\overline{\nu }(x)=ae^{-ax}$ (see \cite{SHI},
p.304). The density of the inverse subordinator can be obtained explicitely
as follows:%
\begin{equation}
l_{g}(x,t)=\int_{0}^{t}h_{g}(z,x)\nu
(t-z)dz=\int_{0}^{t}h_{g}(z,x)e^{-a(t-z)}dz,  \label{ll}
\end{equation}%
where the subordinator's density can be evaluated by inverting its Laplace
transform:%
\begin{equation*}
\widetilde{h}_{g}(\eta ,x)=e^{-x\frac{\eta }{\eta +a}}=e^{-x}e^{\frac{xa}{%
\eta +a}}=e^{-x}+e^{-x}\sum_{j=1}^{\infty }\frac{(xa)^{j}}{j!(\eta +a)^{j}}.
\end{equation*}%
Thus
\begin{equation*}
h_{g}(z,x)=e^{-x}\delta (z)+\frac{e^{-x-az}}{z}\sum_{j=1}^{\infty }\frac{%
(xaz)^{j}}{j!(j-1)!}
\end{equation*}%
which is the transition density of the compound Poisson with exponential
jumps $\sum_{n=1}^{N(t)}X_{j}$, $t>0,$ where $N(t),t>0$ is the Poisson
process with unitary intensity and $X_{j}$ are exponential i.i.d. random
variables with parameter $a$ for any $j=1,2,...$ (see \cite{AAA}, for
details). By substituting in (\ref{ll}) and by standard calculus, we get%
\begin{equation*}
l_{g}(x,t)=e^{-x-at}\sum_{j=0}^{\infty }\frac{(xat)^{j}}{(j!)^{2}}%
=e^{-x-at}W_{1,1}(xat),
\end{equation*}%
where $W_{\beta ,\gamma }(x)=\sum_{j=0}^{\infty }\frac{x^{\beta j}}{j!\Gamma
(\beta j+\gamma )},$ $x,\beta ,\gamma \in \mathbb{C},$ is the Wright
function (see \cite{KIL}, p.54). \ Therefore $\widetilde{l}_{g}(\gamma ,t)$
is given by%
\begin{equation}
\widetilde{l}_{g}(\gamma ,t)=\frac{e^{-at\frac{\gamma }{\gamma +1}}}{\gamma
+1}  \label{ll2}
\end{equation}%
and the time transform in Def.\ref{Def.25},\ i.e. $\mathcal{T}_{g}(t)=\frac{a}{%
2(\gamma +1)}t+2\gamma \log (\gamma +1),$ is surprisingly linear in $t.$ As
a consequence, we can conclude that the generalized Fokker-Planck
equation (\ref{rel3}), with the operator $\mathcal{L}_{t}^{g}$ in Def.\ref{Def.22} specialized with%
\begin{equation*}
D_{t}^{g}u(t)=\int_{0}^{t}\frac{d}{ds}u(t-s)e^{-as}ds,
\end{equation*}%
is satisfied by the transition density of a standard OU process under a
linear combination of the time argument.

Finally, since (\ref{ll2}) is clearly positive definite, the stationary
version of the previous process is the centered Gaussian process with
covariance
\begin{equation*}
r(s)=\frac{\mathcal{\theta }}{\gamma }\frac{e^{-\frac{a\gamma }{\gamma +1}%
|s|}}{\gamma +1},
\end{equation*}%
for any $s\in \mathbb{R},$ $\gamma ,\theta >0.$
\end{example}

\qquad

\end{document}